\documentclass[a4paper,12pt]{article}
\usepackage{parskip}
\usepackage{fullpage}
\usepackage{abstract}

\usepackage{titlesec}
\def\sectionfont{\sffamily\Large\bfseries\boldmath}
\def\subsectionfont{\sffamily\large\bfseries\boldmath}
\def\paragraphfont{\sffamily\normalsize\bfseries\boldmath}
\titleformat*{\section}{\sectionfont}
\titleformat*{\subsection}{\subsectionfont}
\titleformat*{\subsubsection}{\paragraphfont}
\titleformat*{\paragraph}{\paragraphfont}
\titleformat*{\subparagraph}{\paragraphfont}

\usepackage{amsmath,amssymb,mathtools,amsfonts}
\usepackage[english]{babel}
\usepackage[titletoc,title]{appendix}
\usepackage{enumitem}
\setlist{nolistsep}
\usepackage{hyperref}
\hypersetup{
  colorlinks=true,
  linktoc=all,
  linkcolor=black,
  citecolor= black,
  urlcolor=blue
}

\newcommand{\mcf}{\mathcal}
\newcommand{\mbb}{\mathbb}

\newcommand{\Nat}{\mbb{N}}

\newcommand{\Prop}{Prop.}
\newcommand{\Def}{Def.}
\newcommand{\Cor}{Cor.}
\newcommand{\Thm}{Thm.}
\newcommand{\Lem}{Lem.}

\DeclareMathOperator*{\argmin}{\operatorname{argmin}}
\DeclareMathOperator{\Id}{Id}
\DeclareMathOperator{\dom}{dom}
\DeclareMathOperator{\cldom}{\mathop{\overline{dom}}}

\DeclareMathOperator{\fix}{Fix}
\DeclareMathOperator{\prox}{Prox}
\DeclareMathOperator{\range}{ran}
\DeclareMathOperator{\clrange}{\mathop{\overline{ran}}}

\newcommand{\vp}{v_\mcf{P}}
\newcommand{\vd}{v_\mcf{D}}

\newcommand{\indicator}[1]{\iota_{#1}}
\newcommand{\project}[1]{P_{#1}}
\newcommand{\normalCone}[1]{N_{#1}}
\newcommand{\polar}[1]{{#1}^\ominus}

\newcommand{\recession}[1]{{\rm rec}\,{#1}}
\newcommand{\support}[1]{\sigma_{#1}}
\newcommand{\closure}[1]{\overline{#1}}

\newcommand{\norm}[1]{\lVert#1\rVert}
\newcommand{\half}{\tfrac{1}{2}}
\newcommand{\innerprod}[2]{\left\langle{#1}\mid{#2}\right\rangle}

\newcommand{\seq}[1]{({#1})_{n\in\Nat}}
\newcommand{\setB}{B}
\newcommand{\setC}{C}
\newcommand{\setD}{D}

\usepackage{amsthm}
\newtheoremstyle{exampstyle}
  {.5\baselineskip}
  {\topsep}
  {}
  {}
  {\bfseries}
  {.}
  {.5em}
  {}
\theoremstyle{exampstyle}

\newtheorem{theorem}{Theorem}[section]
\newtheorem{corollary}[theorem]{Corollary}
\newtheorem{fact}[theorem]{Fact}
\newtheorem{lemma}[theorem]{Lemma}
\newtheorem{proposition}[theorem]{Proposition}
\newtheorem{remark}[theorem]{Remark}

\title{\bfseries\sffamily On the Minimal Displacement Vector of the Douglas-Rachford Operator}
\author{Goran Banjac}

\begin{document}

\maketitle

\begin{abstract}
The Douglas-Rachford algorithm can be represented as the fixed point iteration of a firmly nonexpansive operator.
When the operator has no fixed points, the algorithm's iterates diverge, but the difference between consecutive iterates converges to the so-called minimal displacement vector, which can be used to certify infeasibility of an optimization problem.
In this paper, we establish new properties of the minimal displacement vector, which allow us to generalize some existing results.
\end{abstract}

\section{Introduction}\label{sec:intro}

The Douglas-Rachford algorithm is a powerful method for minimizing the sum of two convex functions that found applications in numerous research areas including signal processing \cite{Combettes:2007}, machine learning \cite{Boyd:2011}, and control \cite{Stathopoulos:2016}.
The asymptotic behavior of the algorithm is well understood when the problem has a solution.
While there exist some results studying feasibility problems involving two convex sets that do not intersect \cite{Bauschke:2016a,Bauschke:2016b,Bauschke:2017}, some recent works also study a more general setting in which the asymptotic behavior of the algorithm is characterized via the so-called \emph{minimal displacement vector}.
The authors in \cite{Bauschke:2016c} characterize this vector in terms of the domains of the functions, whose sum is to be minimized, and their Fenchel conjugates.
This characterization is used in \cite{Ryu:2019} to show that a nonzero minimal displacement vector implies either primal or dual infeasibility of the problem, but there is an additional assumption imposed, which excludes the case of simultaneous primal and dual infeasibility.
The authors in \cite{Bauschke:2020} derive a new convergence result on the algorithm applied to the problem of minimizing a convex function subject to a linear constraint, but they assume that the Fenchel dual problem is feasible.
The analysis in \cite{Banjac:2019,Banjac:2021} covers the case of simultaneous primal and dual infeasibility for a restricted class of problems and shows that the minimal displacement vector can be decomposed as the sum of two orthogonal vectors, one of which is a certificate of primal infeasibility, and the other of dual infeasibility.

In this paper, we show that the orthogonal decomposition of the minimal displacement vector of the Douglas-Rachford operator established in \cite{Banjac:2019,Banjac:2021} holds in the general case as well.
We also show that the algorithm generates certificates of both primal and dual strong infeasibility.
This allows us to recover the results reported in \cite{Banjac:2019,Banjac:2021} as a special case of our analysis.

The paper is organized as follows.
We introduce some definitions and notation in the remainder of Section~\ref{sec:intro}, and some known results on the Douglas-Rachford algorithm in Section~\ref{sec:dra}.
Section~\ref{sec:displacement} presents a decomposition of the minimal displacement vector and new convergence results.
Finally, Section~\ref{sec:quadratic} applies these new results to the problem of minimizing a convex quadratic function subject to convex constraints.

\subsection{Notation}
All definitions introduced here are standard and can be found in \cite{Bauschke:2017:book}, to which we also refer for basic results on convex analysis and monotone operator theory.

Let $\Nat$ denote the set of nonnegative integers, and $\mcf{H}$, $\mcf{H}_1$, $\mcf{H}_2$ be finite-dimensional real Hilbert spaces with inner products $\innerprod{\cdot}{\cdot}$, induced norms $\norm{\,\cdot\,}$, and identity operators $\Id$.
The power set of $\mcf{H}$ is denoted by $2^\mcf{H}$.
Let $\setD$ be a nonempty subset of $\mcf{H}$ with $\closure{\setD}$ being its \emph{closure}.
We denote the \emph{range} of operator $T\colon\setD\to\mcf{H}$ by $\range T$ and define its \emph{fixed point set} as $\fix T = \lbrace x\in\setD \mid Tx=x \rbrace$.
The \emph{kernel} of a linear operator $A$ is denoted by $\ker{A}$.
For a proper lower semicontinuous convex function $f\colon\mcf{H}\to \left]-\infty,+\infty\right]$, we define its:
\begin{align*}
  &\text{\emph{domain}:} && \dom f = \lbrace x\in\mcf{H} \mid f(x) < +\infty \rbrace, \\
  &\text{\emph{Fenchel conjugate}:} && f^* \colon \mcf{H}\to\left]-\infty,+\infty\right] \colon u\mapsto\sup_{x\in\mcf{H}}\left( \innerprod{x}{u} - f(x) \right), \\
  &\text{\emph{recession function}:} && \recession{f} \colon \mcf{H}\to\left]-\infty,+\infty\right] \colon y\mapsto\sup_{x\in\dom{f}} \left( f(x+y)-f(x) \right), \\
  &\text{\emph{proximity operator}:} && \prox_f \colon \mcf{H}\to\mcf{H} \colon x\mapsto\argmin_{y\in\mcf{H}}\left( f(y) + \half \norm{y-x}^2 \right), \\
  &\text{\emph{subdifferential}:} && \partial f \colon \mcf{H}\to2^\mcf{H} \colon x\mapsto\left\lbrace u\in\mcf{H} \mid (\forall y\in\mcf{H}) \: \innerprod{y-x}{u}+f(x)\le f(y) \right\rbrace.
\end{align*}
For a nonempty closed convex set $\setC\subseteq\mcf{H}$, we define its:
\begin{align*}
  &\text{\emph{polar cone}:} && \polar{\setC} = \Big\lbrace u\in\mcf{H} \mid \sup_{x\in\setC}\innerprod{x}{u} \le 0 \Big\rbrace, \\
  &\text{\emph{recession cone}:} && \recession{\setC} = \left\lbrace x\in\mcf{H} \mid (\forall y\in\setC) \: x+y\in\setC \right\rbrace, \\
  &\text{\emph{indicator function}:} && \indicator{\setC} \colon \mcf{H}\to\left[0,+\infty\right] \colon x\mapsto\begin{cases} 0 & x\in\setC \\ +\infty & \text{otherwise,}\end{cases} \\
  &\text{\emph{support function}:} && \support{\setC} \colon \mcf{H}\to\left]-\infty,+\infty\right] \colon u\mapsto\sup_{x\in\setC} \innerprod{x}{u}, \\
  &\text{\emph{projection operator}:} && \project{\setC} \colon \mcf{H}\to\mcf{H} \colon x\mapsto\argmin_{y\in\setC}\,\norm{y-x}, \\
  &\text{\emph{normal cone operator}:} && \normalCone{\setC} \colon \mcf{H}\to2^\mcf{H} \colon x\mapsto \begin{cases} \big\lbrace u\in\mcf{H} \mid \sup\limits_{y\in\setC} \innerprod{y-x}{u}\le 0 \big\rbrace & x\in\setC \\ \emptyset & x\notin\setC. \end{cases}
\end{align*}

\section{Douglas-Rachford Algorithm}\label{sec:dra}

The Douglas-Rachford algorithm can be used to solve composite minimization problems of the form
\begin{equation}\label{eqn:primal}
  \underset{x\in\mcf{H}}{\rm minimize} \quad f(x) + g(x),
  \tag{$\mcf{P}$}
\end{equation}
where $f\colon\mcf{H}\to\left]-\infty,+\infty\right]$ and $g\colon\mcf{H}\to\left]-\infty,+\infty\right]$ are proper lower semicontinuous convex functions.
We say that \eqref{eqn:primal} is feasible if $0\in\dom f - \dom g$ and strongly infeasible if $0\notin\closure{\dom f - \dom g}$.
The Fenchel dual of \eqref{eqn:primal} can be written as
\begin{equation}\label{eqn:dual}
  \underset{\nu\in\mcf{H}}{\rm minimize} \quad f^*(\nu) + g^*(-\nu).
  \tag{$\mcf{D}$}
\end{equation}
Starting from some $s_0\in\mcf{H}$, the Douglas-Rachford algorithm applied to \eqref{eqn:primal} generates the following iterates:
\begin{subequations}\label{eqn:dra}
\begin{align}
  x_n			    &= \prox_f s_n \label{eqn:dra:x} \\
  \nu_n       &= s_n - x_n \label{eqn:dra:nu} \\
  \tilde{x}_n &= \prox_g ( 2x_n - s_n ) \\
  s_{n+1}	    &= s_n + \tilde{x}_n - x_n, \label{eqn:dra:ss}
\end{align}
\end{subequations}
which can be written compactly as $s_n=T^n s_0$, where
\[
  T = \half\Id + \half(2\prox_g-\Id)(2\prox_f-\Id)
\]
is a firmly nonexpansive operator \cite{Lions:1979}.
It is easy to show from \eqref{eqn:dra} that for all $n\in\Nat$
\[
  s_n - T s_n \in (\dom f - \dom g) \cap (\dom f^* + \dom g^*).
\]
Note that $T$ has a fixed point if and only if $0\in\range(\Id-T)$.
The following fact shows that the sequence $\seq{s_n-Ts_n}$ converges regardless of the existence of a fixed point of~$T$.
\begin{fact}\label{fact:min_displ_vec}
Let $s_0\in\mcf{H}$, $s_n=T^n s_0$, and $v\in\mcf{H}$ be the \emph{minimal displacement vector} of $T$ defined as
\[
  v=\project{\clrange(\Id-T)}(0).
\]
Then
\begin{enumerate}[label=(\roman*)]
  \item \label{fact:min_displ_vec:dlts}  $s_n-s_{n+1} \to v$.
  \item \label{fact:min_displ_vec:v_dom} $v=\project{\closure{\dom f-\dom g}\cap\closure{\dom f^*+\dom g^*}}(0)$.
\end{enumerate}
\end{fact}
\begin{proof}
The first result is \cite[\Cor~2.3]{Baillon:1978} and the second is \cite[\Cor~6.5]{Bauschke:2016c}.
\end{proof}

Since $v$ is defined via the projection onto the set $\clrange(\Id-T)$, which is nonempty closed convex \cite[\Lem~4]{Pazy:1971}, it always exists and must be unique.

\begin{remark}
It is Fact~\ref{fact:min_displ_vec}\ref{fact:min_displ_vec:v_dom}, which relies on \cite{Bauschke:2016c}, that prompted us to work in a finite-dimensional space.
\end{remark}

\section{Minimal Displacement Vector}\label{sec:displacement}

Motivated by the characterization of the minimal displacement vector given in Fact~\ref{fact:min_displ_vec}\ref{fact:min_displ_vec:v_dom} and the decomposition given in \cite[\Prop~2.3]{Bauschke:2020}, we define vectors
\[
  \vp = \project{\closure{\dom f-\dom g}}(0)
  \quad\text{ and }\quad
  \vd = \project{\closure{\dom f^*+\dom g^*}}(0).
\]
Using \Lem~\ref{lem:closures}, we can equivalently characterize these vectors as
\begin{equation}\label{eqn:vp_vd}
  \vp = \project{\closure{\cldom f-\cldom g}}(0)
  \quad\text{ and }\quad
  \vd = \project{\closure{\cldom f^*+\cldom g^*}}(0).
\end{equation}

\subsection{Static Results}\label{subsec:static}

Although it is obvious that nonzero $\vp$ and $\vd$ imply strong infeasibility of \eqref{eqn:primal} and \eqref{eqn:dual}, respectively, we next provide some useful identities.
\begin{proposition}\label{prop:vp_vd_rec_fcn}
Vectors $\vp$ and $\vd$ satisfy the following equalities:
\begin{align*}
  \recession{f^*}(-\vp) + \recession{g^*}(\vp) &= -\norm{\vp}^2 \\
  \recession{f}(-\vd) + \recession{g}(-\vd) &= -\norm{\vd}^2.
\end{align*}
\end{proposition}
\begin{proof}
Since proofs of both equalities follow very similar arguments, we only provide a proof for the first.
Using the definition of $\vp$ and \cite[\Prop~6.47]{Bauschke:2017:book}, we have
\[
  -\vp\in\normalCone{\closure{\dom f-\dom g}}(\vp).
\]
Using \cite[\Thm~16.29]{Bauschke:2017:book} and the facts that $\indicator{\setD}^*=\support{\setD}$ and $\partial\indicator{\setD}=\normalCone{\setD}$, the inclusion above is equivalent to
\[
  -\norm{\vp}^2 = \support{\closure{\dom f-\dom g}}(-\vp) = \support{\dom f}(-\vp) + \support{\dom g}(\vp) = \recession{f^*}(-\vp) + \recession{g^*}(\vp),
\]
where the second equality follows from $\support{\closure{\setC+\setD}}=\support{\setC+\setD}=\support{\setC}+\support{\setD}$ and $\support{-\setC}=\support{\setC} \circ (-\Id)$, and the third from \cite[\Prop~13.49]{Bauschke:2017:book}.
\end{proof}

\begin{proposition}\label{prop:vp_vd_v}
The following relations hold between vectors $\vp$, $\vd$, and $v$:
\begin{enumerate}[label=(\roman*)]
  \item\label{prop:vp_vd_v:i}
  $-\vp \in \polar{\left( \recession(\cldom f) \right)} \cap \polar{\left( \recession(-\cldom g) \right)}$.

  \item\label{prop:vp_vd_v:ii}
  $-\vd \in \polar{\left( \recession(\cldom f^*) \right)} \cap \polar{\left( \recession(\cldom g^*) \right)}$.

  \item\label{prop:vp_vd_v:iii}
  $-\vp\in\recession(\cldom f^*)\cap\recession(-\cldom g^*)$.

  \item\label{prop:vp_vd_v:iv}
  $-\vd\in\recession(\cldom f)\cap\recession(\cldom g)$.

  \item\label{prop:vp_vd_v:v}
  $\innerprod{\vp}{\vd} = 0$.

  \item\label{prop:vp_vd_v:vi}
  $\vp + \vd \in \closure{\dom f - \dom g} \cap \closure{\dom f^* + \dom g^*}$.

  \item\label{prop:vp_vd_v:vii}
  $v = \vp + \vd$.
\end{enumerate}
\end{proposition}
\begin{proof}
\ref{prop:vp_vd_v:i}\&\ref{prop:vp_vd_v:ii}:
Follow from \eqref{eqn:vp_vd} and \cite[\Cor~2.7]{Bauschke:2004}.

\ref{prop:vp_vd_v:iii}\&\ref{prop:vp_vd_v:iv}:
Follow from parts \ref{prop:vp_vd_v:i}\&\ref{prop:vp_vd_v:ii} and \Lem~\ref{lem:rec_dom}.

\ref{prop:vp_vd_v:v}:
Since $-\vp \in \polar{\left( \recession(\cldom f) \right)}$ and $-\vd\in\recession(\cldom f)$, we have $\innerprod{\vp}{\vd} \le 0$.
Also, since $-\vp \in \polar{\left( \recession(-\cldom g) \right)}$ and $-\vd\in\recession(\cldom g)$, we have $\innerprod{\vp}{\vd} \ge 0$.
Therefore, it must be that $\innerprod{\vp}{\vd} = 0$.

\ref{prop:vp_vd_v:vi}:
By \ref{prop:vp_vd_v:iv}, we have $-\vd\in\recession(\cldom g)$, hence
\begin{align*}
  \vp + \vd \in \closure{\cldom f - \cldom g} + \vd = \closure{\cldom f - (\cldom g - \vd)} &\subseteq \closure{\cldom f - \cldom g} \\
  &= \closure{\dom f - \dom g},
\end{align*}
where the last equality follows from \Lem~\ref{lem:closures}.
Similarly, by \ref{prop:vp_vd_v:iii} we have $\vp\in\recession(\cldom g^*)$, hence
\begin{align*}
  \vp + \vd \in \vp + \closure{\cldom f^* + \cldom g^*} = \closure{\cldom f^* + (\cldom g^* + \vp)} &\subseteq \closure{\cldom f^* + \cldom g^*} \\
  &= \closure{\dom f^* + \dom g^*}.
\end{align*}

\ref{prop:vp_vd_v:vii}:
Assuming that $\vp+\vd=0$, the identity follows from Fact~\ref{fact:min_displ_vec}\ref{fact:min_displ_vec:v_dom} and part~\ref{prop:vp_vd_v:vi}.
We next assume that $\vp+\vd \ne 0$.
Using \cite[\Thm~3.16]{Bauschke:2017:book} together with the definitions of $\vp$, $\vd$, and $v$, we have
\begin{align*}
  \innerprod{v-\vp}{-\vp} \le 0 \quad &\Longleftrightarrow\quad \norm{\vp}^2 \le \innerprod{v}{\vp} \\
  \innerprod{v-\vd}{-\vd} \le 0 \quad &\Longleftrightarrow\quad \norm{\vd}^2 \le \innerprod{v}{\vd},
\end{align*}
which together with part~\ref{prop:vp_vd_v:v} implies
\[
  \norm{\vp+\vd}^2 = \norm{\vp}^2 + \norm{\vd}^2 \le \innerprod{v}{\vp+\vd} \le \norm{v} \norm{\vp+\vd}.
\]
Dividing the inequality by $\norm{\vp+\vd} \ne 0$, we get $\norm{\vp+\vd}\le\norm{v}$.
Using part~\ref{prop:vp_vd_v:vi} and the fact that $v$ is the unique element of minimum norm in $\closure{\dom f - \dom g} \cap \closure{\dom f^* + \dom g^*}$, we obtain the result.
\end{proof}

\begin{corollary}\label{cor:vp_vd}
The following relations hold between vectors $v$, $\vp$, and $\vd$:
\begin{enumerate}[label=(\roman*)]
  \item $-\vp = \project{\polar{(\recession(\cldom f))}}(-v)$.
  \item $-\vd = \project{\recession(\cldom f)}(-v)$.
\end{enumerate}
\end{corollary}
\begin{proof}
Follows directly from \Prop~\ref{prop:vp_vd_v} and \cite[\Cor~6.31]{Bauschke:2017:book}.
\end{proof}

The authors in \cite{Ryu:2019} have also established connections between recession functions and the minimal displacement vector, but the equalities in \Prop~\ref{prop:vp_vd_rec_fcn} provide a tight characterization of the left-hand sides and improve the bounds given in~\cite{Ryu:2019}.
Also, if problem~\eqref{eqn:primal} is feasible, then $\vp=0$, which according to \Prop~\ref{prop:vp_vd_v}\ref{prop:vp_vd_v:vii} implies $v=\vd$; similarly, if problem~\eqref{eqn:dual} is feasible, then $v=\vp$.
Although these implications were established in \cite{Ryu:2019}, they follow as a special case of our analysis, which is also applicable when both \eqref{eqn:primal} and \eqref{eqn:dual} are infeasible.

\subsection{Dynamic Results}\label{subsec:dynamic}

Fact~\ref{fact:min_displ_vec}\ref{fact:min_displ_vec:dlts} shows that the difference between consecutive iterates of the so-called \emph{governing sequence} $\seq{s_n}$ always converges.
We next show that the same holds for the \emph{shadow sequence} $\seq{x_n}$.

\begin{theorem}\label{thm:limits}
Let $s_0\in\mcf{H}$ and $\seq{x_n,\tilde{x}_n,\nu_n}$ be the sequences generated by~\eqref{eqn:dra}.
Then
\[
  (x_n-x_{n+1},\tilde{x}_n-\tilde{x}_{n+1},\nu_n-\nu_{n+1})\to(\vd,\vd,\vp).
\]
\end{theorem}
\begin{proof}
Using Moreau's decomposition \cite[\Thm~14.3(ii)]{Bauschke:2017:book}, it is easy to show from~\eqref{eqn:dra} that for all $n\in\Nat$
\begin{subequations}
\begin{align}
  x_n - x_{n+1} &= \prox_{f^*} s_{n+1} + \prox_{g^*}(2x_n-s_n) \in \dom f^* + \dom g^* \label{eqn:diff_x_incl} \\
  \nu_n - \nu_{n+1} &= \prox_f s_{n+1} - \prox_g(2x_n-s_n) \in \dom f - \dom g. \label{eqn:diff_nu_incl}
\end{align}
\end{subequations}
From the definitions of $\vp$ and $\vd$, and the inclusions above, it follows that
\begin{subequations}
\begin{align}
  \norm{\vd} &\le \underline{\lim} \, \norm{x_n - x_{n+1}} \label{eqn:diff_x_bound} \\
  \norm{\vp} &\le \underline{\lim} \, \norm{\nu_n - \nu_{n+1}}. \label{eqn:diff_nu_bound}
\end{align}
\end{subequations}
Since $\prox_f$ is firmly nonexpansive \cite[\Prop~12.28]{Bauschke:2017:book}, \cite[\Def~4.1(i)]{Bauschke:2017:book} implies
\[
  \norm{s_n-s_{n+1}}^2 \ge \norm{x_n-x_{n+1}}^2 + \norm{\nu_n-\nu_{n+1}}^2, \quad \forall n\in\Nat.
\]
Taking the limit superior of the inequality above, we get
\begin{align*}
  \lim \, \norm{s_n-s_{n+1}}^2 &\ge \overline{\lim} \, \left( \norm{x_n-x_{n+1}}^2 + \norm{\nu_n-\nu_{n+1}}^2 \right) \\
  &\ge \overline{\lim} \, \norm{x_n-x_{n+1}}^2 + \underline{\lim} \, \norm{\nu_n-\nu_{n+1}}^2,
\end{align*}
and thus
\[
  \overline{\lim} \, \norm{x_n-x_{n+1}}^2 \le \lim \, \norm{s_n-s_{n+1}}^2 - \underline{\lim} \, \norm{\nu_n-\nu_{n+1}}^2 \le \norm{v}^2 - \norm{\vp}^2 = \norm{\vd}^2,
\]
where the second inequality follows from Fact~\ref{fact:min_displ_vec}\ref{fact:min_displ_vec:dlts} and \eqref{eqn:diff_nu_bound}, and the equality from \Prop~\ref{prop:vp_vd_v}\ref{prop:vp_vd_v:v}\&\ref{prop:vp_vd_v:vii}.
Combining the inequality above with \eqref{eqn:diff_x_bound} yields $\norm{x_n - x_{n+1}} \to \norm{\vd}$.
Using the inclusion in \eqref{eqn:diff_x_incl} and the fact that $\vd$ is the unique element of minimum norm in $\closure{\dom f^*+\dom g^*}$, it follows that $x_n-x_{n+1}\to\vd$; $\tilde{x}_n-\tilde{x}_{n+1}\to\vd$ and $\nu_n-\nu_{n+1}\to\vp$ then follow directly from \eqref{eqn:dra}, Fact~\ref{fact:min_displ_vec}\ref{fact:min_displ_vec:dlts}, and \Prop~\ref{prop:vp_vd_v}\ref{prop:vp_vd_v:vii}.
\end{proof}

\begin{corollary}
Let $s_0\in\mcf{H}$ and $\seq{x_n,\tilde{x}_n,\nu_n}$ be the sequences generated by~\eqref{eqn:dra}.
Then
\[
  -\tfrac{1}{n}(x_n,\tilde{x}_n,\nu_n)\to(\vd,\vd,\vp).
\]
\end{corollary}
\begin{proof}
Follows directly from \Thm~\ref{thm:limits} and the fact that, given a sequence $\seq{a_n}$ in $\mcf{H}$, $a_n\to a$ implies $\tfrac{1}{n}\sum_{i=1}^n a_n \to a$.
\end{proof}

The results above show that the strong infeasibility certificates $\vp$ and $\vd$ can be obtained as the limits of sequences constructed from the Douglas-Rachford iterates.

\section{Constrained Minimization of a Quadratic Function}\label{sec:quadratic}

Consider the following convex optimization problem:
\begin{equation}\label{eqn:quad_primal}
  \begin{array}{ll}
    \underset{z\in\setB}{\rm minimize} & \half \innerprod{z}{Qz} + \innerprod{q}{z} \\
    \textrm{subject to} & Az\in\setC,
  \end{array}
\end{equation}
with $Q\colon\mcf{H}_1\to\mcf{H}_1$ a monotone self-adjoint linear operator, $q\in\mcf{H}_1$, $A\colon\mcf{H}_1\to\mcf{H}_2$ a linear operator, and $\setB$ and $\setC$ nonempty closed convex subsets of $\mcf{H}_1$ and $\mcf{H}_2$, respectively.
The objective function of the problem is convex, continuous, and Fr\'echet differentiable \cite[\Prop~17.36(i)]{Bauschke:2017:book}.

The following proposition is a direct extension of \cite[\Prop~3.1]{Banjac:2019}.
\begin{proposition}\label{prop:infeas}~
\begin{enumerate}[label=(\roman*)]
  \item If there exists a pair $(\bar{\lambda},\bar{\mu})\in\mcf{H}_1\times\mcf{H}_2$ such that $\bar{\lambda}+A^*\bar{\mu}=0$ and $\support{\setB}(\bar{\lambda})+\support{\setC}(\bar{\mu})<0$, then problem~\eqref{eqn:quad_primal} is strongly infeasible.
  \item If there exists a $\bar{z}\in\recession{\setB}$ such that $Q\bar{z}=0$, $A\bar{z}\in\recession{\setC}$, and $\innerprod{q}{\bar{z}}<0$, then the dual of problem~\eqref{eqn:quad_primal} is strongly infeasible.
\end{enumerate}
\end{proposition}

Observe that \eqref{eqn:quad_primal} is an instance of problem~\eqref{eqn:primal} with $f\colon\mcf{H}_1\times\mcf{H}_2 \to \left]-\infty,+\infty\right]$ and $g\colon\mcf{H}_1\times\mcf{H}_2 \to \left]-\infty,+\infty\right]$ given by
\begin{subequations}\label{eqn:quad_fcn}
\begin{align}
  f(z,y) &= \indicator{\setB}(z) + \indicator{\setC}(y) \label{eqn:quad_fcn:f} \\
  g(z,y) &= \half \innerprod{z}{Qz} + \innerprod{q}{z} + \indicator{Az=y}(z,y), \label{eqn:quad_fcn:g}
\end{align}
\end{subequations}
where $\indicator{Az=y}$ denotes the indicator function of the set $\{(z,y)\in\mcf{H}_1\times\mcf{H}_2 \mid Az=y\}$.
Due to \Lem~\ref{lem:g_conjugate}, $f^*\colon\mcf{H}_1\times\mcf{H}_2 \to \left]-\infty,+\infty\right]$ and $g^*\colon\mcf{H}_1\times\mcf{H}_2 \to \left]-\infty,+\infty\right]$ are given by
\begin{subequations}\label{eqn:quad_fcn_conj}
\begin{align}
  f^*(\lambda,\mu) &= \support{\setB}(\lambda) + \support{\setC}(\mu) \\
  g^*(\lambda,\mu) &= \half \innerprod{\lambda+A^*\mu-q}{Q^\dagger (\lambda+A^*\mu-q)} + \indicator{\range{Q}}(\lambda+A^*\mu-q).
\end{align}
\end{subequations}
We next consider iteration~\eqref{eqn:dra} applied to the problem of minimizing the sum of the functions given in \eqref{eqn:quad_fcn}.

When $\setB=\mcf{H}_1$ and $\setC$ has some additional structure, problem~\eqref{eqn:quad_primal} reduces to the one considered in \cite{Banjac:2019}, where the Douglas-Rachford algorithm (which is equivalent to the alternating direction method of multipliers) was shown to generate certificates of primal and dual strong infeasibility.
This result was generalized in \cite{Banjac:2021} to the case where $\setC$ is an arbitrary nonempty closed convex set.
We next show that these results are a direct consequence of our analysis presented in Section~\ref{sec:displacement}.
We use the notation
\[
  v=(v',v''),
  \qquad
  \vp=(\vp',\vp''),
  \qquad
  \vd=(\vd',\vd''),
\]
where the first and second components are elements of $\mcf{H}_1$ and $\mcf{H}_2$, respectively.

\begin{proposition}\label{prop:quad_dra}
Let $f\colon\mcf{H}_1\times\mcf{H}_2 \to \left]-\infty,+\infty\right]$ and $g\colon\mcf{H}_1\times\mcf{H}_2 \to \left]-\infty,+\infty\right]$ be given by \eqref{eqn:quad_fcn}, and $(z_n,y_n)$ and $(\lambda_n,\mu_n)$ be the Douglas-Rachford iterates corresponding to $x_n$ and $\nu_n$ in \eqref{eqn:dra}, respectively.
Then
\begin{enumerate}[label=(\roman*)]
  \item \label{prop:quad_dra:vd}   $(-\vd',-\vd'') = \big( \project{\recession{\setB}}(-v'), \project{\recession{\setC}}(-v'') \big)$.
  \item \label{prop:quad_dra:vp}   $(-\vp',-\vp'') = \big( \project{\polar{(\recession{\setB})}}(-v'), \project{\polar{(\recession{\setC})}}(-v'') \big)$.
  \item \label{prop:quad_dra:dlt}  $(z_n-z_{n+1},y_n-y_{n+1},\lambda_n-\lambda_{n+1},\mu_n-\mu_{n+1}) \to (\vd',\vd'',\vp',\vp'')$.
  \item \label{prop:quad_dra:Qvd}  $Q\vd'=0$.
  \item \label{prop:quad_dra:Avd}  $A\vd'=\vd''$.
  \item \label{prop:quad_dra:qvd}  $\innerprod{q}{-\vd'} = -\norm{\vd}^2$.
  \item \label{prop:quad_dra:Atvp} $\vp'+A^*\vp''=0$.
  \item \label{prop:quad_dra:supp} $\support{\setB}(-\vp')+\support{\setC}(-\vp'')=-\norm{\vp}^2$.
\end{enumerate}
\end{proposition}
\begin{proof}
\ref{prop:quad_dra:vd}\&\ref{prop:quad_dra:vp}:
Follow from \Cor~\ref{cor:vp_vd} with $\cldom f = \setB\times\setC$.

\ref{prop:quad_dra:dlt}:
Follows from \Thm~\ref{thm:limits}.

\ref{prop:quad_dra:Qvd}\&\ref{prop:quad_dra:Avd}\&\ref{prop:quad_dra:qvd}:
Using the identity $\recession{f}=\support{\dom f^*}$ \cite[\Prop~13.49]{Bauschke:2017:book}, it is easy to show that the recession functions of those in \eqref{eqn:quad_fcn} are given by
\begin{align*}
  \recession{f}(\bar{z},\bar{y}) &= \indicator{\recession{\setB}}(\bar{z}) + \indicator{\recession{\setC}}(\bar{y}) \\
  \recession{g}(\bar{z},\bar{y}) &= \innerprod{q}{\bar{z}} + \indicator{\ker{Q}}(\bar{z}) + \indicator{Az=y}(\bar{z},\bar{y}).
\end{align*}
Due to \Prop~\ref{prop:vp_vd_rec_fcn}, we have
\[
  -\norm{\vd}^2 = \indicator{\recession{\setB}}(-\vd') + \indicator{\recession{\setC}}(-\vd'') + \innerprod{q}{-\vd'} + \indicator{\ker{Q}}(-\vd') + \indicator{Az=y}(-\vd',-\vd''),
\]
which implies
\[
  Q\vd' = 0,
  \qquad
  A\vd'=\vd'',
  \qquad
  \innerprod{q}{-\vd'} = -\norm{\vd}^2.
\]

\ref{prop:quad_dra:Atvp}\&\ref{prop:quad_dra:supp}:
Using the identity $\recession{f^*}=\support{\dom f}$, it is easy to show that the recession functions of those in \eqref{eqn:quad_fcn_conj} are given by
\begin{align*}
  \recession{f^*}(\bar{\lambda},\bar{\mu}) &= \support{\setB}(\bar{\lambda}) + \support{\setC}(\bar{\mu}) \\
  \recession{g^*}(\bar{\lambda},\bar{\mu}) &= \indicator{\{0\}}(\bar{\lambda}+A^*\bar{\mu}).
\end{align*}
Due to \Prop~\ref{prop:vp_vd_rec_fcn}, we have
\[
  -\norm{\vp}^2 = \support{\setB}(-\vp') + \support{\setC}(-\vp'') + \indicator{\{0\}}(\vp'+A^*\vp''),
\]
which implies
\[
  \vp'+A^*\vp''=0,
  \qquad
  \support{\setB}(-\vp') + \support{\setC}(-\vp'')=-\norm{\vp}^2.
  \qedhere
\]
\end{proof}

\Prop~\ref{prop:infeas} and \Prop~\ref{prop:quad_dra} imply that, if $\vp$ is nonzero, then problem~\eqref{eqn:quad_primal} is strongly infeasible, and similarly, if $\vd$ is nonzero, then its dual is strongly infeasible.
When $\setB=\mcf{H}_1$, the expressions in \Prop~\ref{prop:quad_dra} reduce to those given in \cite{Banjac:2019,Banjac:2021} since $\recession{\setB}=\mcf{H}_1$ implies $\vd'=v'$, $\vp'=0$, $\support{\setB}(-\vp')=0$, and $\norm{\vp}=\norm{\vp''}$.

\section*{Acknowledgements}
The author would like to thank Angeliki Kamoutsi and Liviu Aolaritei for helpful discussions, and John Lygeros for editorial suggestions on the manuscript.
This project has received funding from the European Research Council (ERC) under the European Union's Horizon 2020 research and innovation programme grant agreement OCAL, No.\ 787845.

\appendices
\section{Supporting Results}

\begin{lemma}\label{lem:closures}
Let $\setC\subseteq\mcf{H}$ and $\setD\subseteq\mcf{H}$ be nonempty sets.
Then
\[
  \closure{\setC-\setD} = \closure{\closure{\setC}-\closure{\setD}}.
\]
\end{lemma}
\begin{proof}
We first show that
\begin{equation}\label{eqn:incl}
  \closure{\setC}-\closure{\setD} \subseteq \closure{\setC-\setD}.
\end{equation}
Let $x\in\closure{\setC}-\closure{\setD}$, that is, $x=c-d$ for some $c\in\closure{\setC}$ and $d\in\closure{\setD}$.
Then, there exist sequences $\seq{c_n}$ in $\setC$ and $\seq{d_n}$ in $\setD$ such that $c_n\to c$ and $d_n\to d$.
Hence, the sequence $\seq{x_n}$, given by $x_n=c_n-d_n$, lies in $\setC-\setD$ and $x_n\to x$, thus, $x\in\closure{\setC-\setD}$.

Now, we have
\[
  \setC-\setD \subseteq \closure{\setC}-\closure{\setD} \subseteq \closure{\setC-\setD},
\]
where the first inclusion is straightforward as $\setC\subseteq\closure{\setC}$ and $\setD\subseteq\closure{\setD}$, and the second is given in \eqref{eqn:incl}.
Taking the closures of these sets gives
\[
  \closure{\setC-\setD} \subseteq \closure{\closure{\setC}-\closure{\setD}} \subseteq \closure{\setC-\setD},
\]
which implies $\closure{\setC-\setD}=\closure{\closure{\setC}-\closure{\setD}}$.
\end{proof}

\begin{lemma}\label{lem:rec_dom}
Let $f\colon\mcf{H}\to\left]-\infty,+\infty\right]$ be a proper lower semicontinuous convex function.
Then
\[
  \polar{\left( \recession(\cldom f) \right)} = \cldom\support{\cldom f} = \cldom(\recession{f^*}) \subseteq \recession(\cldom f^*).
\]
\end{lemma}
\begin{proof}
The first equality can be found in \cite{Adly:2004} and the second is \cite[\Prop~13.49]{Bauschke:2017:book}.
To show the last inclusion, let $d\in\dom(\recession{f^*})$.
Then $\recession{f^*}(d) < +\infty$, which implies
\begin{align*}
  (\forall y\in\dom f^*) \: f^*(y+d) < +\infty \quad \Longleftrightarrow & \quad (\forall y\in\dom f^*) \: y+d\in\dom f^* \\
  \Longrightarrow & \quad (\forall z\in\cldom f^*) \: z+d\in\cldom f^* \\
  \Longleftrightarrow & \quad d\in\recession(\cldom f^*).
\end{align*}
To prove the implication above, let $z\in\cldom f^*$.
Then, there exists a sequence $\seq{y_n}$ in $\dom f^*$ such that $y_n\to z$.
Hence, the sequence $\seq{y_n+d}$ lies in $\dom f^*$ and converges to $z+d$, thus, $z+d\in\cldom f^*$.

Therefore, it follows that $\dom(\recession{f^*})\subseteq\recession(\cldom f^*)$.
Moreover, since $\recession(\cldom f^*)$ is always closed, taking the closures of these sets gives $\cldom(\recession{f^*})\subseteq\recession(\cldom f^*)$.
\end{proof}

\begin{lemma}\label{lem:g_conjugate}
Let $g\colon\mcf{H}_1\times\mcf{H}_2 \to \left]-\infty,+\infty\right]$ be given by \eqref{eqn:quad_fcn:g}.
Then its Fenchel conjugate $g^*\colon\mcf{H}_1\times\mcf{H}_2 \to \left]-\infty,+\infty\right]$ is given by
\[
  g^*(\lambda,\mu) = \half \innerprod{\lambda+A^*\mu-q}{Q^\dagger (\lambda+A^*\mu-q)} + \indicator{\range{Q}}(\lambda+A^*\mu-q).
\]
where $Q^\dagger$ is the Moore-Penrose inverse of $Q$.
\end{lemma}
\begin{proof}
The Fenchel conjugate of the quadratic function $h\colon\mcf{H}_1\to \left]-\infty,+\infty\right[ \colon z \mapsto \half \innerprod{z}{Qz} + \innerprod{q}{z}$ is given by
\[
  h^*(\lambda) = \sup_{z\in\mcf{H}_1} \left( \innerprod{\lambda}{z} - \half \innerprod{z}{Qz} - \innerprod{q}{z} \right) = \half \innerprod{\lambda-q}{Q^\dagger (\lambda-q)} + \indicator{\range{Q}}(\lambda-q),
\]
which follows directly from \cite[\Prop~13.23(iii) \& \Prop~17.36(iii)]{Bauschke:2017:book}.
Thus, the Fenchel conjugate of $g$ is given by
\begin{align*}
  g^*(\lambda,\mu) &= \sup_{(z,y)\in\mcf{H}_1\times\mcf{H}_2} \left( \innerprod{\lambda}{z} + \innerprod{\mu}{y} - \half \innerprod{z}{Qz} - \innerprod{q}{z} - \indicator{Az=y}(z,y) \right) \\
  &= \sup_{z\in\mcf{H}_1} \left( \innerprod{\lambda+A^*\mu}{z} - \half \innerprod{z}{Qz} - \innerprod{q}{z} \right) \\
  &= h^*(\lambda+A^*\mu).
  \qedhere
\end{align*}
\end{proof}

\bibliography{refs}

\end{document}